\documentclass[12pt]{amsart}
\usepackage{geometry} 
\usepackage{graphicx}
\usepackage{amsfonts}
\usepackage{amsthm}
\usepackage{amsmath}
\usepackage{amssymb}
\usepackage[arrow,matrix,curve,color]{xy}
\usepackage{color}
\usepackage[normalem]{ulem}
\usepackage{tikz}
\usepackage{xspace}
\definecolor{light-blue}{rgb}{0.8,0.85,1}
\definecolor{light-red}{rgb}{1,.4,.4}
\definecolor{purp}{rgb}{.7,.3,1}
\definecolor{yel}{rgb}{1,1,.5}
\definecolor{cy}{rgb}{0,1,1}
\usepackage{colortbl}
\usepackage{multirow}
\usepackage{array}
\newtheorem{theorem}{Theorem}

\theoremstyle{definition}

\newtheorem{definition}[theorem]{Definition}

\newtheorem{remark}[theorem]{Remark}

\newcommand{\co}{\colon\,}

\newcommand{\bT}{\mathbb T}
\newcommand{\bR}{\mathbb R}
\newcommand{\bC}{\mathbb C}

\newcommand{\bZ}{\mathbb Z}
\newcommand{\bP}{\mathbb P}

\newcommand{\cA}{\mathcal A}

\newcommand{\cE}{\mathcal E}

\newcommand{\cH}{\mathcal H}

\newcommand{\cO}{\mathcal O}

\newcommand{\cS}{\mathcal S}

\newcommand{\fg}{\mathfrak g}

\newcommand{\Th}{\Theta}

\newcommand{\U}{\mathop{\rm U}}

\newcommand{\GL}{\mathop{\rm GL}}
\newcommand{\Sp}{\mathop{\rm Sp}}

\newcommand{\lp}{\textup{(}}
\newcommand{\rp}{\textup{)}}

\newcommand{\End}{\operatorname{End}}

\newcommand{\Ch}{\operatorname{Ch}}
\newcommand{\Tr}{\operatorname{Tr}}

\renewcommand\Re{\operatorname{Re}}
\renewcommand\Im{\operatorname{Im}}

\newcommand{\nc}{noncommutative}

\begin{document}
\title[Noncommutative Riemann-Roch]{The Riemann-Roch Theorem\\
on higher dimensional\\ complex noncommutative tori}
\author{Varghese Mathai}
\address{Department of Pure Mathematics\\
School of Mathematical Sciences\\
University of Adelaide\\
Adelaide, SA 5005, Australia} 
\email[Varghese Mathai]{mathai.varghese@adelaide.edu.au}
\author{Jonathan Rosenberg}
\address{Department of Mathematics\\
University of Maryland\\
College Park, MD 20742-4015, USA} 
\email[Jonathan Rosenberg]{jmr@math.umd.edu}
\thanks{VM partially supported by
ARC grants FL170100020 and DP170101054.  
JR partially supported by NSF grant DMS-1607162.  
JR also thanks VM and the University of Adelaide for their
hospitality during a research visit in July 2019.}
\begin{abstract}
  We prove analogues of the Riemann-Roch Theorem and the Hodge
  Theorem for {\nc} tori (of any dimension) equipped with
  complex structures, and discuss implications for the
  question of how to distinguish ``{\nc} abelian varieties''
  from ``non-algebraic'' {\nc} complex tori.
\end{abstract}
\keywords{noncommutative torus, abelian variety, index theory,
  Riemann-Roch Theorem, Hodge Theorem}
\subjclass[2010]{Primary 46L87; Secondary 19L10
  58B34 58J42 53C55 57R20 14K20}

\maketitle

\section*{Introduction}
\label{sec:intro}

There is by now a quite extensive literature on {\nc} differential
geometry on {\nc} tori, and a much smaller literature on {\nc}
complex analytic geometry for {\nc} complex tori, which is
still in its infancy except for the case where the complex dimension is
$1$, in which case most of the obvious problems have been settled
by work of Schwarz and Polishchuk
\cite{MR1977884,MR2054986,MR2330589,MR2247860}.

The purpose of this paper is to study {\nc} complex tori of
arbitrary complex dimension, and in particular to try to find
{\nc} analogues of the Riemann-Roch Theorem, the Hodge Theorem
for the Dolbeault complex, and the Riemann characterization of
abelian varieties within the class of all complex tori.
Formulation of the main  theorems is done in Section \ref{sec:setup}.
In Section \ref{sec:mainproof}, we give the proofs of the
Riemann-Roch Theorem and the Hodge Theorem.  In Section
\ref{sec:cases}, we examine various special cases and give
{\nc} analogues (admittedly still imperfect) for the existence of
non-algebraic complex tori in complex dimension $>1$ and for
the characterization of abelian varieties via Riemann forms.

\section{General Setup}
\label{sec:setup}

Let $\Th$ be a skew-symmetric real matrix, say of size $d \times d$.  The 
\emph{noncommutative torus} $A_\Th$ is the universal $C^*$-algebra with 
$d$ unitary generators $U_j$, $1\le j\le d$, subject to the basic commutation
relation
\[
U_jU_k = e^{2\pi i \Theta_{jk}} U_kU_j.
\]
This algebra carries a \emph{gauge action} of $\bT^d$ via 
\[
t\cdot \left(U_1^{n_1}\cdots U_d^{n_d} \right)
= t_1^{n_1}\cdots t_d^{n_d} U_1^{n_1}\cdots U_d^{n_d} , \quad
t = (t_1, \cdots, t_d)\in \bT^d.
\]
There are associated infinitesimal generators $\delta_j$, 
which are $*$-derivations (really defined on
the smooth subalgebra $\cA_\Th$, which we will introduce shortly),
with $\delta_j$ sending $U_j$ to $2\pi i U_j$ and sending the other
$U_k$, $k\ne j$, to $0$.  Together, these define a tangent space or
(commutative) Lie algebra $\fg = \text{span}\,(\delta_1,\cdots, \delta_d)$.
The algebra $A_\Th$ also carries a canonical tracial state $\Tr$ invariant 
under the gauge action, sending $1$ to $1$ and sending a monomial
$U_1^{n_1}\cdots U_d^{n_d}$ to $0$ unless all of the $n_j$ vanish.
Because of the commutation relation, any element of
$A_\Th$ has a canonical (formal) expansion in terms of
the monomials $U_1^{n_1}\cdots U_d^{n_d}$.  This formal expansion
may not converge in the $C^*$-norm, but there is a canonical
\emph{smooth subalgebra} $\cA_\Th$ consisting of elements of the form
$\sum_{n\in \bZ^d} c_n U^n$ (here $U^n$ is a shorthand for
$U_1^{n_1}\cdots U_d^{n_d}$), for which the coefficients
$c_n$ are rapidly decreasing, i.e., form a sequence in the Schwartz
space $\cS(\bZ^d)$.  The smooth subalgebra $\cA_\Th$ is exactly the
algebra of $C^\infty$ vectors for the gauge action of $\bT^d$, and plays
the role of the algebra of $C^\infty$ functions (with which it
coincides if $\Th=0$). Since we will want to
use methods of differential geometry, we will almost always work 
with $\cA_\Th$ in place of its $C^*$-completion $A_\Th$.
\begin{definition}
\label{def:complexstr}
Let $A_\Th$ and $\cA_\Th$ be as above, and assume that the ``dimension''
$d=2n$ is even.  We will refer to $\fg = \bR\delta_1+\cdots +
\bR\delta_{2n}$ as the \emph{tangent space}.
A \emph{complex structure} on $A_\Th$ and $\cA_\Th$
is a choice of an endomorphism $J$ of $\fg$ satisfying $J^2 = -1$.
It thus defines an isomorphism $\fg_\bC\cong \fg^{\text{hol}}
\oplus \fg^{\text{antihol}}$
as a direct sum of holomorphic and antiholomorphic tangent spaces,
namely the $\pm i$-eigenspaces of $J$.  There is a similar splitting
of the complexified cotangent space $\fg_\bC^*$.
The pair $(A_\Th, J)$ will be called a \emph{{\nc} complex torus}
of complex dimension $n$.
\end{definition}
If $J$ is as in Definition \ref{def:complexstr}, then the standard $n$-torus
$\bT^{2n}=\bR^{2n}/\bZ^{2n}$, together with the complex structure
given by $J$, can equivalently be viewed as $\bC^n/\Lambda$,
with the standard complex structure on $\bC^n$ (given by multiplication by $i$),
identified with the $+i$-eigenspace of $J$,
but with $\Lambda$ now a ``skewed'' lattice in $\bC^n$ given by
$\Lambda=Q\cdot\bZ^{2n}$, where $Q$ is an $n\times 2n$ matrix
with values in $\bC$.  Then a famous classical theorem, due basically
to Riemann, though not rigorously proved until later in the 19th 
century\footnote{\cite{MR3525914} gives a detailed discussion of the history.}, is
\begin{theorem}[{Riemann Period Relations \cite[Theorem 2.3]{MR3525914}}]
\label{thm:Riemann}
The complex torus $\bC^n/\Lambda$, $\Lambda=Q\cdot\bZ^{2n}$, is a 
complex abelian variety if and only if $Q$ satisfies the Riemann conditions,
i.e., after suitable change of basis, it can be put in the form $(\Omega, I_n)$
with $\Omega\in M_n(\bC)$ symmetric with positive-definite imaginary part.
\end{theorem}
Note that the condition of Theorem \ref{thm:Riemann} always holds in
complex dimension $1$, since any lattice in $\bC$ can always be
brought into the form $\bZ + \tau \bZ$, $\Re\tau>0$, after multiplying
by a nonzero complex number to send one of the basis vectors to $1$.
In higher dimensions, since the space of $J$'s can be identified
with $\GL(2n,\bR)/\GL(n, \bC)$, which has real dimension $2n^2$
or complex dimension $n^2$, whereas the Siegel upper half-space of
symmetric matrices with positive-definite imaginary part can be identified
with $\Sp(2n, \bR)/\U(n)$, which has real dimension $n^2+n$ and complex 
dimension $\frac12(n^2 + n)$ (strictly smaller than $n^2$ for $n > 1$),
\emph{not every complex torus is an abelian variety}.

There is an equivalent version of Theorem \ref{thm:Riemann} which is
sometimes more convenient:
\begin{theorem}[{Riemann Period Relations, alternate version
      \cite[Ch.\ I, \S3]{MR2514037},
      \cite[Theorem 2.8]{milneAV}, and \cite[Theorem 3.1]{MR3525914}}]
\label{thm:Riemann1}
Let $X=\bC^n/\Lambda$ be a complex torus.  Then $X$ is a complex
abelian variety if and only if it admits a Riemann form, i.e., there
is a nondegenerate skew-symmetric bilinear form
$E\co \Lambda\times \Lambda\to \bZ$ whose extension $E_\bR$
to a skew-symmetric real bilinear form on $\bC^n$
satisfies $E_\bR(iv, iw) = E_\bR(v, w)$ and such that
the associated hermitian form
$H(v,w) = E_\bR(iv,w) + iE_\bR(v,w)$ is positive definite.
\end{theorem}

We will be interested in determining the extent to which 
something similar to Theorem \ref{thm:Riemann} or Theorem \ref{thm:Riemann1}
holds in the noncommutative world.  The difficulty, of course, is that there
is no accepted definition of a noncommutative abelian variety, or clarity
about whether various possible definitions are equivalent or not.
This problem was studied
already in \cite{MR1827023,MR2329879,KimKim,MR2525268}, using different
methods and slightly different definitions, but none of the approaches
has yet completely solved the problem.

Given a complex structure $J$ on $A_\Th$, one has a Hodge splitting of
$\bigwedge \fg^*_\bC$ as a direct sum of pieces
\[
{\textstyle\bigwedge}^{p,q} \fg^*_\bC = {\textstyle\bigwedge}^p \left(\fg^*_\bC\right)^{\text{hol}}
\otimes_\bC {\textstyle\bigwedge}^q \left(\fg^*_\bC\right)^{\text{antihol}}
\]
and we define
$\Omega^{p,q}(\cA_\Th) = \cA_\Th \otimes \bigwedge^{p,q} \fg^*_\bC$.
The \emph{Dolbeault complex} of $\cA_\Th$ is
\[
\cA_\Th = \Omega^{0,0}(\cA_\Th) \xrightarrow{\bar\partial}
\Omega^{0,1}(\cA_\Th) \xrightarrow{\bar\partial} \cdots
\xrightarrow{\bar\partial} \Omega^{0,n}(\cA_\Th),
\]
where the $\bar\partial$ operator is defined by 
\[
\bar\partial(a) = \sum_j \bar\partial_j(a)\otimes d\bar z_j,
\]
where $\{\bar\partial_j\}_{j=1,\cdots, n}$ is a 
basis for $\fg^{\text{antihol}}$ and 
$\{d\bar z_j\}_{j=1,\cdots, n}$ is the dual basis of
$\left(\fg^{\text{antihol}}\right)^*$.  We extend this as usual
to a map
$\Omega^{p,q}(\cA_\Th)\to \Omega^{p,q+1}(\cA_\Th)$.  Note that
this is a special case of the exterior differential calculus
defined by Connes in \cite{MR572645}, and we will use much
of Connes' theory in what follows.

Now we repeat a definition from \cite{MR1977884}:
\begin{definition}
\label{def:holobundle}
A \emph{vector bundle} $E$ over $\cA_\Th$ will mean a finitely generated
projective (right) module.  (This is consistent with the usual
notion of vector bundle because of the Serre-Swan Theorem.)
A \emph{holomorphic vector bundle} $E$ will mean such a bundle
equipped with a \emph{holomorphic connection} $\overline\nabla$,
meaning a map $E\to E \otimes \left(\fg^{\text{antihol}}\right)^*$
satisfying the Leibniz rule
\begin{equation}
\label{eq:Leibniz}
\overline\nabla_{\bar\partial_j}(e\cdot a) =
\overline\nabla_{\bar\partial_j}(e)\cdot a +
e \cdot \bar\partial_j(a).
\end{equation}
Note that any vector bundle can be equipped with a
holomorphic connection simply by writing
$E = p (\cA_\Th)^r$ for some projection $p$, and then
defining $\overline\nabla = p(\bar\partial)^r$.
If $\overline\nabla$ and $\overline\nabla'$
are two holomorphic connections,
then when one subtracts one from the other,
the second term on the right in \eqref{eq:Leibniz}
cancels, and thus the space of holomorphic connections
for a fixed projective module $E$ is a principal
homogeneous space for 
$\End_{A_\Th}(E)\otimes \left(\fg^{\text{antihol}}\right)^*$. 
Note that $\End_{A_\Th}(E)$ is a Morita equivalent
{\nc} torus, which we can think of as acting on the
\emph{left}.

But most of the time we will be interested in
\emph{flat holomorphic connections}, that is, ones
satisfying the flatness condition that one gets a
\emph{Dolbeault complex}
\begin{equation}
\label{eq:Dolbeault}  
E = \Omega^{0,0}(E) \xrightarrow{\overline\nabla}
\Omega^{0,1}(E) \xrightarrow{\overline\nabla} \cdots
\xrightarrow{\overline\nabla} \Omega^{0,n}(E),
\end{equation}
i.e., such that $(\overline\nabla)^2=0$.  The cohomology
groups of \eqref{eq:Dolbeault} will be denoted
$H^j(E)$.  (We suppress mention of the choice of flat holomorphic
connection, as this usually will be understood.)
For example, when $E$ is a rank-one free module with the
connection given by $\bar\partial$ itself,
$H^j(E)\cong \bigwedge^j\bC^n$ has dimension $\binom n j$,
and $H^0(E)$ consists precisely of the constants $\bC$.

Incidentally, the flatness condition $(\overline\nabla)^2=0$
is a new and nontrivial condition when $n>1$.  As a result, in
higher dimensions, it is not always immediately clear if a given
projective module admits a \emph{flat} holomorphic connection or not.
In \cite{MR941652}, Rieffel shows that if $\Theta$ is irrational,
i.e., contains at least one irrational entry, then every
projective module is ``standard'' and admits a standard connection
with constant curvature, i.e., the curvature is just
given by a(n explicit)
linear functional on $\bigwedge^2\fg$.
If this linear functional is of Hodge type $(1,1)$,\footnote{This
  condition should be familiar from K\"ahler geometry.}
so it has no $(0,2)$-component, then the antiholomorphic
part of the standard connection is a flat holomorphic connection.
Explicit calculations in \cite{KimKim,MR2525268} show
that this may or may not be the case.

We need the flatness condition to define cohomology, since
a {\nc} torus does not have ``points''\footnote{In fact, for
  generic $\Th$, $A_\Th$ is a simple algebra.}
and thus we cannot
use sheaf cohomology.  However, $H^0(E)=\ker\overline\nabla$,
the space of
``holomorphic sections,'' is defined for any holomorphic
bundle, whether or not the connection is flat.
\end{definition}
\begin{definition}
\label{def:index}
Let $(E, \overline \nabla)$ be a holomorphic bundle in the
sense of Definition \ref{def:holobundle}.  Its \emph{index}
or \emph{Euler characteristic}
$\chi(E)$ is defined to be the index of the
operator
\[
\overline\nabla + (\overline\nabla)^*\co
\Omega^{0,\text{even}}(E) \to \Omega^{0,\text{odd}}(E),
\]
which is an elliptic operator in the sense of
\cite{MR572645}.  Note that this is well-defined whether
or not $\overline\nabla$ is flat.  The adjoint
$(\overline\nabla)^*$ here is defined using an appropriate
Hilbert space completion of $\Omega^{\bullet,\bullet}(E)$,
which will be explained in detail in the proofs of the main theorems
in Section \ref{sec:mainproof}.  The Hilbert space inner product is
not unique, but the non-uniqueness won't affect the value of the index.
\end{definition}  

The following main theorems are in some sense ``known'' but we have
not found an explicit reference for them, except in complex dimension $1$,
where they appear as \cite[Proposition 2.5]{MR1977884} and
\cite[Theorem 2.8]{MR2054986}.
\begin{theorem}[Riemann-Roch Theorem]
\label{thm:RR}
Let $J$ be a holomorphic structure on $A_\Th$, and let
$E$ be a holomorphic vector
bundle over $(\cA_\Th, J)$.  Then
\[
\chi(E) = \Ch_{\text{\textup{top}}}([E]).
\]
\end{theorem}

Here the right-hand side of the formula is defined as follows.
The vector bundle $E$ defines a class 
$[E]\in K_0(\cA_\Th)\cong K_0(A_\Th)$
{\lp}the equality of the two $K$-groups 
follows immediately from 
\cite[Lemme 1]{MR572645}{\rp}.  This group is free abelian;
in fact $K_\bullet(A_\Th)$ is sent isomorphically
under the Chern character to the exterior algebra
$\bigwedge^\bullet  H^1(\bT^{2n})=\bigwedge^\bullet \bZ^{2n}$.
{\lp}See \cite{MR731772} and \cite{MR3742588}.  The denominators
in the Chern character cancel out for basically the same reasons
as in \cite[Proposition 4.3]{hatcherVB}.{\rp}
The notation $\Ch_{\text{top}}([E])$ means the component
of the Chern character in the top-degree summand
$\bigwedge^{2n}\bZ^{2n}\cong \bZ$.  In the case $n=1$
studied by Connes, Polishchuk, and Schwarz, this reduces
to what would usually be called the Chern class
$c_1([E])$, or what Polishchuk and Schwarz call the degree.
\begin{theorem}[Hodge Theorem]
\label{thm:Hodge}
If, in the situation of \textup{Theorem \ref{thm:RR}},
the holomorphic connection on $E$ is flat, so that the cohomology
groups $H^j(E)$ are defined, then these groups are finite-dimensional,
and $\chi(E)=\sum_{j=0}^n (-1)^j\dim H^j(E)$.
\end{theorem}
\begin{remark}
\label{rem:rigidity}
We would like to point one that Theorem \ref{thm:RR} contains within
it an extraordinary rigidity result.  Namely, the Euler 
characteristic $\chi(E)$
is \textbf{independent of the complex structure $J$ and of the
holomorphic structure on $E$, as well as of the noncommutativity
parameter $\Theta$} (within a family of holomorphic bundles
for different {\nc} tori with varying $\Theta$).

However, the actual cohomology groups in Theorem \ref{thm:Hodge}
do indeed depend on this data, as it is easy to see from
\cite{MR2054986} that there are cases (with complex dimension $n=1$)
where one gets cancellation in this formula, whereas by
\cite[Proposition 2.5]{MR1977884}, for ``standard'' holomorphic bundles
with nonzero $c_1$ (i.e., nonzero degree), this does not happen,
and only one of the two cohomology groups is nonzero.  The paper
\cite{MR2247860} studies jumps in the cohomology groups
$H^j(E)$ (for $E$ a rank-one free module) in more detail
in the case $n=1$.

The rigidity of the index shows that {\nc} tori are indeed very special.
Within the class of compact complex manifolds, one can find examples
of diffeomorphic manifolds (with different complex structures) and 
\emph{different} values of the Todd genus, which is the Euler characteristic
(or index)
of the structure sheaf $\cO$.  This was first observed by Hirzebruch
for almost-complex manifolds diffeomorphic to $\bC\bP^3$
\cite[\S2.1]{MR0066013} and improved by Kotschick
\cite{MR2855094}, who gave examples that are actually 
diffeomorphic (but necessarily non-birational) complex
projective varieties.  Of course this cannot happen with complex tori
since they always have vanishing Chern classes.
\end{remark}

The next section, Section \ref{sec:mainproof},
contains the proofs of the main theorems,
Theorem \ref{thm:RR} and Theorem \ref{thm:Hodge}.  The last section,
Section \ref{sec:cases}, contains discussion of specific cases,
as well as Theorems \ref{thm:splittingcase}, \ref{thm:nonalg},
and \ref{thm:NCRiemann}, which are {\nc} replacements for
the Riemann conditions, Theorems \ref{thm:Riemann} and
\ref{thm:Riemann1}. 

\section{Proofs of the Riemann-Roch and Hodge Theorems}
\label{sec:mainproof}

In this section we give the proofs of the two main theorems,
Theorem \ref{thm:RR} and Theorem \ref{thm:Hodge}.

\begin{proof}[Proof of Theorem \ref{thm:RR}]
We begin with the case where $E=(\cA_\Th)^r$ is a free module over $\cA_\Th$,
in which case (as a Fr\'echet space) it can be identified
with $\cS(\bT^{2n})^r$.  If we take $\overline\nabla$ to be
just $(\bar\partial)^r$, then (forgetting the {\nc} structure)
$\overline\nabla$ gives rise to the usual Dolbeault complex
on the Schwartz space for a translation-invariant 
complex structure $J$ and flat holomorphic connection on $\bT^{2n}$.
Choose a $J$-invariant Euclidean structure on $\bR^{2n}$;
this will give rise to a translation-invariant K\"ahler metric
on $(\bT^{2n},J)$.
Thus the cohomology groups are $H^j(E)\cong \bigl(\bigwedge^j\bC^n\bigr)^r$
(really one should use the conjugate space
$\bigl(\bigwedge^j\overline\bC^n\bigr)^r$,
for the reasons explained in \cite[\S1]{MR2514037})
and the operator $D=\overline\nabla+ \overline\nabla^*$ from
even to odd forms is Fredholm, and 
$T = \begin{pmatrix} 0 & D\\ D^* & 0\end{pmatrix}$ is essentially
self-adjoint (on $L^2(\bT^{2n})^r$ for our choice of Riemannian metric).
By the usual Hodge Theorem, the index $\chi(E)=0$.

If we keep $E$ a free module but change the holomorphic 
connection to $\overline\nabla'$, the new connection differs from
the old one by a multiplication operator of order $0$, while $D$ is an
elliptic differential
operator of order $1$.  So the new operator
$T' = \begin{pmatrix} 0 & D'\\ D'^* & 0\end{pmatrix}$
differs from $T$ by a lower-order perturbation, and the index remains 
unchanged.

Next, suppose we are given a holomorphic connection 
$\overline\nabla_E$ on $E$, as well as a complementary bundle
$F$ with a holomorphic connection $\overline\nabla_F$, so that the
direct sum $E\oplus F$ is free.  Then 
$\overline\nabla = \overline\nabla_E \oplus \overline\nabla_F$ is a 
holomorphic connection for a free module, so by the cases we've already
considered, $\ker\bigl( \overline\nabla + \overline\nabla^* \bigr)$ is
finite-dimensional and the index of the operator from even to odd forms
is $0$.  It follows that $D=\overline\nabla_E+ \overline\nabla_E^*$ from
even to odd forms is Fredholm, and 
$T = \begin{pmatrix} 0 & D\\ D^* & 0\end{pmatrix}$ is essentially
self-adjoint on the Hilbert space completion of $E$.  In particular,
the index $\chi(E)$ is well-defined, and independent of the choice of 
connection (by the independence result for free modules).
We also obtain that $\chi(E)+\chi(F)=0$.

In fact, we also obtain that $\chi(E)$ is independent of the 
complex structure $J$, since as we vary
the $J$ continuously, we get a continuous family of
Fredholm operators $T_J$, and thus the index cannot change.

With all of these facts at our disposal, we see that
$\chi$ is a well-defined homomorphism
$K_0(A_\Th)\to \bZ$.  To prove the index theorem, it
suffices to compute this map on a set of generators for $K_0$.

As noted in \cite[Proof of Proposition 2.6]{MR2054986},
the theorem can now be reduced to the commutative case, i.e.,
to the Hirzebruch Riemann-Roch Theorem for complex manifolds.
The reason is that Elliott \cite{MR731772} gives us a canonical isomorphism
of $K_0(A_\Th)$ with the free abelian group $K^0(\bT^{2n})$, 
that varies continuously
(and thus remains constant) as $\Theta$ changes. First we compute the
homomorphism $\chi$ in the commutative case. Choose
a finite collection of vector bundles $E$ on $\bT^{2n}$ generating
$K^0$.  For each choice of $E$, we have (since the Todd
polynomial of any complex torus is just the constant $1$)
by Hirzebruch Riemann-Roch
\[
\chi(E) = \int_{\bT^{2n}} \Ch(E) \cdot\text{Td}(\bT^{2n})
= \Ch_{\text{\textup{top}}}([E]).
\]
So the theorem holds in the commutative case.

Now given the bundle $E$ over $\bT^{2n}$, the homotopy
$t\mapsto t\Th$, $t\in [0, 1]$ gives us a 
homotopy $E_t$ of vector bundles (so that $E_t$ is a 
vector bundle for $A_{t\Th}$ and the $K_0$ class of $E_t$
doesn't change under the Elliott parametrization of $K_0$).
In this way we again get a continuous family $T_t$ of Fredholm operators
(on the same Hilbert space $L^2(E)$, the Hilbert space of $L^2$ sections
of $E$ with respect to Lebesgue measure on the base
and some hermitian metric on the bundle), and so the index doesn't change.
Thus the theorem holds in general.
\end{proof}

\begin{remark}
There is one little subtlety which we should note here, which 
is that in the commutative case, not every positive
$K_0$-class is represented by a vector bundle.  (It may
only be represented by a formal difference of vector
bundles.  This was noted in \cite[p.\ 320]{MR941652}.) 
So if one starts with a holomorphic
bundle for a {\nc} complex torus, after perturbing
$\Theta$ to $0$ it can happen that the bundle
degenerates (i.e., no longer gives an actual bundle).
This is not such a problem since one can take care
to use $K_0$ generators that correspond to actual
vector bundles even in the commutative case.
(It suffices to jack up the rank by adding a free
module in order to get into the stable range.)
\end{remark}

\begin{proof}[Proof of Theorem \ref{thm:Hodge}]
Suppose $(A_\Th, J)$ is a {\nc} complex torus, and let
$\overline\nabla$ be a flat holomorphic connection on a smooth vector
bundle $E$.  We basically follow the method of proof of the classical
Hodge Theorem, using the machinery developed in
\cite[Proposition 6.5]{MR0212836} and \cite{MR1174159}.
The main technical issue is that $\overline\nabla$ is \emph{a priori}
defined on $E$, which inherits a Fr\'echet space structure
from $\cA_\Th$, but we will want to use analysis of
operators on Hilbert spaces, so we need Hilbert space
completions for the spaces $\Omega^{0,j}(E)$.  For this
we use $L^2(\Tr)$, the Hilbert space completion of
$A_\Th$ in the GNS representation
of the tracial state $\Tr$ (this will usually be a $\mathrm{II}_1$ factor
representation).  In the commutative case $\Th=0$,
$L^2(\Tr)$ simply becomes $L^2(\bT^{2n})$ with respect to
normalized Lebesgue measure, so this is something familiar.
Since $E$ is a direct summand in $(\cA_\Th)^r$ for some $r$,
we get a Hilbert space completion $\cH$ of $E$, namely the
closure of $E$ in $L^2(\Tr)^r$, and
Hilbert spaces $\cH_j$ completing $\Omega^{0,j}(E)$,
with $\cH_0=\cH$.

Next, we observe that 
$\overline\nabla\co \Omega^{0,j}(E)\to \Omega^{0,j+1}(E)$
is closable;  this follows from the case of a free module with
an arbitrary holomorphic connection,  since $\bar\partial$ 
looks like a constant-coefficient first-order differential
operator acting on the Schwartz space $\cS(\bT^{2n})$, 
and on a free module, $\overline\nabla$
differs from $\bar\partial^r$ by a zero-order operator. We let 
$d^j\co \cH_j\to \cH_{j+1}$ denote the closed extension of
$\overline\nabla$.  Then
\[
\cH_0 \xrightarrow{d^0} \cH_1 \xrightarrow{d^1} \cdots
\xrightarrow{d^n}  \cH^n
\]
(note that the operators are closed but only densely defined)
is a Hilbert complex in the sense of \cite{MR1174159}.

The operator $D=\overline\nabla + \overline\nabla^*$
(sending even forms to odd forms and \emph{vice versa})
is elliptic, in the sense that if we combine it with the corresponding
operator coming from a holomorphic connection (not necessarily
flat) on a complementary vector bundle, we get an
elliptic (vector-valued) first-order operator on $\cS(\bT^{2n})$.
Thus it's essentially self-adjoint on $\bigoplus\cH_j$.
So all the theory of \cite[\S\S2--3]{MR1174159}
immediately applies (see in particular
Theorem 2.12 and Theorem 3.5), the Hilbert complex is Fredholm
(with the same cohomology as the smooth complex), 
and the Hodge Theorem follows.
\end{proof}

\section{Examination of Some Specific Cases}
\label{sec:cases}

\subsection{The Product Case}
\label{sec:product}

The simplest kinds of complex higher-dimensional {\nc}  tori
are what one can call ``products of {\nc} elliptic curves.'' These are the {\nc}
analogues of products of elliptic curves, which are very special even
within the class of abelian varieties, let alone within all complex tori.
For example, when $n=2$, products of two elliptic curves vary within a moduli
space of complex dimension $2$, while abelian varieties of complex dimension
$2$ vary within a moduli space of complex dimension $3$, and complex tori
vary in a moduli space of complex dimension $4$.

To be more precise, we make a definition.
\begin{definition}
\label{def:producttype}
A {\nc} complex torus $A_\Th$ will be said to be of product type if the
$2n\times 2n$ matrices $J$ and $\Th$ both split as block direct sums
of $2\times 2$ matrices of the same form, i.e.,
\[
\Th = \begin{pmatrix} 
\begin{matrix}0 & \theta_1\\-\theta_1 & 0 \end{matrix} & \mathbf{0}_2 & \cdots &  \mathbf{0}_2\\
\mathbf{0}_2 & \begin{matrix}0 & \theta_2\\-\theta_2 & 0\end{matrix} & \cdots &  \mathbf{0}_2\\
\vdots & \vdots & \ddots & \vdots\\
\mathbf{0}_2 & \mathbf{0}_2 &\cdots & \begin{matrix}0 & \theta_n\\-\theta_n & 0\end{matrix}
\end{pmatrix},
\]
and similarly with $J$, except that the diagonal blocks in this case can be
arbitrary $2\times 2$ real matrices $J_j$
with trace $0$ and determinant $1$, $j=1,\cdots, n$.
Note that in the product-type case,
$A_\Th = A_{\theta_1} \otimes \cdots \otimes A_{\theta_n}$,
where $A_\theta$ denotes the {\nc} $2$-torus with parameter
$\theta$, and the complex
structure $J$ also splits as $J_1\otimes \cdots \otimes J_n$.

For simplicity of notation, a complex {\nc} torus with complex dimension $1$
will be called a \emph{{\nc} elliptic curve}.
\end{definition}

In the product-type case, we can carry over the results of \cite{MR1977884}, 
\cite{MR2054986}, \cite{MR2330589}, etc., to
get a complete description of the holomorphic geometry of $(A_\Th, J)$.
As this should be regarded as the ``trivial case'' of
higher-dimensional {\nc} tori,
we do not go into exhaustive detail, but give one representative theorem.
Recall that by \cite{MR2054986}, the holomorphic vector bundles over
a {\nc} elliptic curve are completely classified.  They are formed from
iterated extensions of the standard holomorphic bundles considered in
\cite{MR1977884}.

\begin{theorem}
\label{thm:Hproduct}
Suppose $(A_\Th,J)=(A_{\theta_1},J_1)\otimes \cdots \otimes (A_{\theta_n},J_n)$
is of product type, and let $E_j$ be a holomorphic vector bundle for
the {\nc} elliptic curve $(A_{\theta_j},J_j)$.  Then
$E = E_1\otimes \cdots \otimes E_n$ is a holomorphic vector bundle for
$(A_\Th,J)$ with a flat holomorphic connection,
and its cohomology groups are given by
$H^\bullet(E) = \bigotimes_{j=1}^n H^\bullet(E_j)$.
{\lp}The gradings add, and since $H^\bullet(E_j)$ is concentrated in
degrees $0$ and $1$, $H^j(E)$ is the sum of the tensor products with
$j$ factors in degree $1$.{\rp}
\end{theorem}
\begin{proof}
It is clear that we get a flat holomorphic connection on $E$
by operating separately on the tensor factors.  
The cohomology calculation follows since
the Dolbeault complex for $E$ splits as the
tensor product of the complexes for the {\nc} elliptic curve factors.
\end{proof}

\subsection{The Splitting Case}
\label{sec:product2}

A case which is not quite as simple as that of product type {\nc}
complex tori, but where many of the same considerations apply,
is the case where one has a ``splitting.''  This is the {\nc}
analogue of the following classical situation.  Suppose
$\bC^n/\Lambda$ is a complex torus, and suppose there is
surjective $\bC$-linear map $f\co\bC^n\to \bC$ sending $\Lambda$ to a
lattice in $\bC$.  That means the complex torus $\bC^n/\Lambda$
has an elliptic curve $\bC/f(\Lambda)$ as a quotient, and thus splits
as an extension of this elliptic curve by the complex subtorus
$(\ker f)/\left(\ker f\vert_\Lambda\right)$.  When $n=2$, we thus have an
extension of one elliptic curve by another.  The following
is a fact about classical complex algebraic geometry.
\begin{remark}
\label{rem:splittorus}
Let $\bC^2/\Lambda$ is a complex torus which is an extension of
one elliptic curve by another.  In other words, assume there
exist $\tau,\tau'$ in the upper half-plane and complex linear
maps $\alpha\co \bC\to \bC^2$, $\beta\co \bC^2\to \bC$ such that
\[
0\to \bC \xrightarrow{\alpha} \bC^2 \xrightarrow{\beta} \bC\to 0
\quad \text{and} \quad 
0\to (\bZ+\tau'\bZ) \xrightarrow{\alpha} \Lambda
\xrightarrow{\beta} (\bZ+\tau\bZ)\to 0
\]
are short exact.  Then $\bC^2/\Lambda$ need not be an abelian variety.

Indeed, the hypothesis of the proposition implies that after changing the basis for 
$\bC^2$ by an element of $\GL(2,\bC)$, we can assume $\Lambda$
is spanned by the vectors
\[
e_1 = \begin{pmatrix} 1\\ 0 \end{pmatrix},
e_2 = \begin{pmatrix} \tau' \\ 0 \end{pmatrix},
e_3 = \begin{pmatrix} 0\\ 1 \end{pmatrix},
e_4 = \begin{pmatrix} w\\ \tau \end{pmatrix},
\]
for some $w\in \bC$.  If $w$ were $0$ then clearly we'd have the 
product of two elliptic curves and this would be an abelian variety.
In fact, an extension of abelian varieties within the category of
abelian varieties has to split up to isogeny 
\cite[Proposition 10.1]{milneAV}.  But for sufficiently generic $w$,
it is easy to see that this torus does not split; in fact, 
that $\bC^2/\Lambda$ and the dual torus $(\bC^2)^*/\Lambda^*$
are not isogenous, which implies that we do not have an abelian 
variety.\footnote{Here $(\bC^2)^*$ is the conjugate-linear dual
of $\bC^2$ and $\Lambda^*=\{f\in (\bC^2)^*: \Im f(\lambda)\in \bZ 
\ \forall \lambda\in \Lambda\}$ --- see \cite[I.2]{milneAV}.}
\end{remark}

In spite of Remark \ref{rem:splittorus}, we still get some useful information
in the {\nc} splitting situation.  Suppose $A_\Th$ is a {\nc} torus
with a holomorphic structure that ``splits,'' i.e., such that $J$ is of block
diagonal form, say $\begin{pmatrix} J_1 & 0 \\ 0 & J_2 \end{pmatrix} $,
with $J_1$ of size $2\times 2$.  Then we obtain an obvious
holomorphic injection
\begin{equation}
\label{eq:injection}
\varphi\co
(A_\theta, J_1) \to (A_\Th, J),\quad \text{where }\theta = \Theta_{12},
\end{equation}
given by $u_j\mapsto U_j$, with $u_j$, $j=1,\,2$, the canonical 
generators of $A_\theta$,
which respects smooth structures.  The holomorphic geometry
of the {\nc} elliptic curve $(A_\theta, J_1)$ is completely 
understood, thanks to \cite{MR1977884}, \cite{MR2054986},
and \cite{MR2330589}.

Let $E$ be a holomorphic vector bundle over $\cA_\theta$,
$\theta=\theta_{12}$ and $J_1$ as in \eqref{eq:injection}.
Then $\varphi_*(E)=E\otimes_\varphi \cA_\Th$ will be a holomorphic
vector bundle over $(\cA_\Th, J)$, and knowledge of the 
cohomology of $E$ should give some information about its
cohomology.

To see what one should expect, consider the classical case where $X$
is a complex torus (of complex dimension $n$)
with a (nontrivial) holomorphic group homomorphism
to an elliptic curve $C$. We have a short exact sequence of
complex tori $1\to F\to X\xrightarrow{p} C\to 1$, and if $E$ is a
holomorphic vector bundle over $C$, there is a Leray spectral
sequence 
\[
H^k(C, R^\ell p_* p^*\cE) \Rightarrow H^{k+\ell}(X, p^*\cE).
\]
Here $\cE$ is the sheaf of germs of holomorphic sections of $E$,
whose cohomology is basically what we are calling $H^*(E)$.
By the projection formula, the $E_2$ term simplifies to give a 
Leray-Serre spectral 
sequence 
\[
H^k(C, H^\ell(F,\cO_F)\otimes \cE) = H^k(C, \cE) \otimes H^\ell(F,\cO_F)
\Rightarrow H^{k+\ell}(X, p^*\cE).
\]
The equality on the left follows from the fact that since the fibration is
topologically a product, there is no monodromy; i.e., $\pi_1(C)$
acts trivially on the cohomology of the fiber.  Furthermore, since
$C$ is a curve and $\cO_F$ has cohomology only up to dimension
$n-1$, the spectral sequence is concentrated in the range
$0\le k\le 1,\ 0\le\ell\le n-1$.  In particular there
can't be any differentials,
and the edge homomorphism $H^k(C, \cE) \to H^k(X, p^*\cE)$
is injective.  So we get a lower bound on the size of $H^0(p^*\cE)$.
In fact when $E$ is an ample line bundle (so that $H^k(C, \cE)=0$
for $k\ge 1$), then $\dim H^k(X, p^*\cE) = \deg(E)\binom{n-1}{k}$.

In a similar fashion, in the {\nc} context we have been talking about,
we should get a lower bound on the dimension of $H^0$ for
$\varphi_*(E) = E\otimes_\varphi \cA_\Th$. To see this, let
$\overline\nabla$ be a holomorphic connection on $E$, and
let $\bar\partial_1$ on $\cA_\Th$ correspond to $\bar\partial$
on $A_\theta$.  We define $\overline\nabla'$  on $\varphi_*(E)$
by $\overline\nabla'_1 = \overline\nabla\otimes 1
+ 1 \otimes \bar\partial_1$,
$\overline\nabla'_j = 1\otimes \bar\partial_j$ for $j>1$.
To check that this is consistent, observe that we have
the relation
$ea\otimes b=e\otimes \varphi(a)b$ for $e\in E$,
$a\in \cA_\theta$, $b\in \cA_\Theta$, but
\[
\begin{aligned}
\overline\nabla'_1 (ea\otimes b)& = \overline\nabla(ea)\otimes b
+ ea\otimes \bar\partial_1(b)\\
&= \bigl(\overline\nabla(e)a+e\bar\partial(a)\bigr)\otimes b
+ ea\otimes \bar\partial_1(b)\\
&=\overline\nabla(e)\otimes \varphi(a)b +
e \otimes\varphi(\bar\partial(a))b + e\otimes
\varphi(a)\bar\partial_1(b)\\&=
\bigl(\overline\nabla\otimes 1 + 1 \otimes \bar\partial_1\bigr)(e\otimes
\varphi(a)b) \\
&= \overline\nabla'_1 (e\otimes\varphi(a)b).
\end{aligned}
\]
And similarly, for $j>1$, since $\bar\partial_j$ kills the image
of $\varphi$,
\[
\begin{aligned}
  \overline\nabla'_j (ea\otimes b) &= ea\otimes \bar\partial_j(b)\\
  &= e \otimes \varphi(a)\bar\partial_j(b)\\
  &= e \otimes \bar\partial_j\bigl(\varphi(a)b\bigr)\\
  &= \overline\nabla'_j(e \otimes \varphi(a)b).
\end{aligned}
\]
The Leibniz rule for $\overline\nabla'$ is clear since it holds for
each $\bar\partial_j$.  Finally, we check that the flatness
condition holds.  Suppose $j,\,k>1$. Since $\bar\partial_j$ and
$\bar\partial_k$ commute, so do $\overline\nabla'_j$ and
$\overline\nabla'_k$. And
\[
\begin{aligned}
  \overline\nabla'_j\overline\nabla'_1 (e\otimes b)
  &= \overline\nabla'_j\bigl(\overline\nabla(e)\otimes b
  + e \otimes \bar\partial_1(b)\bigr)\\
  &= \overline\nabla(e)\otimes \bar\partial_j(b) +
  e \otimes \bar\partial_j(\bar\partial_1(b))\\
  &= \overline\nabla(e)\otimes \bar\partial_j(b) +
  e \otimes \bar\partial_1(\bar\partial_j(b))\\
  &= \overline\nabla'_1 (e\otimes \bar\partial_j(b))\\
  &= \overline\nabla'_1\overline\nabla'_j (e\otimes b).
\end{aligned}
\]
Thus $\overline\nabla'$ is indeed a flat holomorphic connection on
$\varphi_*(E)$.  Since it obviously annihilates $e\otimes 1$
for any $e\in \ker \overline\nabla$, we have proved the following:
\begin{theorem}
\label{thm:splittingcase}
Suppose $A_\Th$ is a {\nc} torus
with a holomorphic structure that ``splits,'' i.e., such that $J$ is of block
diagonal form, say $\begin{pmatrix} J_1 & 0 \\ 0 & J_2 \end{pmatrix} $,
with $J_1$ of size $2\times 2$.  Define the holomorphic injection
\[ \varphi\co
(A_\theta, J_1) \to (A_\Th, J),\quad \text{where }\theta = \Theta_{12},
\]
by $u_j\mapsto U_j$, with $u_j$, $j=1,\,2$, the canonical 
generators of $A_\theta$.  Then if $E$ is a holomorphic bundle
for $(\cA_\theta, J_1)$ with holomorphic connection $\overline\nabla$,
there is an induced flat holomorphic connection $\overline\nabla'$
on $\varphi_*(E) = E\otimes_\varphi \cA_\Theta$, and
$\dim H^0(\varphi_*(E)) \ge \dim H^0(E)$.
\end{theorem}

\subsection{The Non-Algebraic Case}
\label{sec:nonalg}

Finally, we prove a result which is an analogue
of the fact that there exists a complex torus $\cA_\Th$ of complex
dimension $2$ with no non-constant meromorphic functions.
(This phenomenon is analyzed in detail in \cite{Wells}.)
\begin{theorem}
\label{thm:nonalg}
There is a noncommutative complex torus of complex dimension $2$
with the property that for any non-free standard projective
module with constant
curvature connection as in \cite{MR941652}, the associated constant
curvature holomorphic connection $\overline\nabla$ is non-flat.
\end{theorem}
\begin{remark}
We would conjecture, though it seems hard to figure out how to prove
this, that one can even arrange for $\cA_\Th$ to admit no non-trivial
``meromorphic functions."
Here a possible definition of meromorphic function is the one from
\cite[\S6]{MR2444094}: it is a
formal quotient $u^{-1}v$,\footnote{Note that $u^{-1}$ is not assumed
  to exist in $A_\Th$.  That's why the quotient is ``formal.''
  However, it makes sense in a suitable localization of the algebra,
  as explained in \cite[\S6]{MR2444094}.}
where $u$ is neither a
left nor right zero divisor, and where $u$ and $v$ satisfy the
Cauchy-Riemann equations
\[
\bar\partial_j(u) = f_ju,\quad \bar\partial_j(v) = f_jv.
\]
\end{remark}
\begin{proof}[Proof of Theorem \ref{thm:nonalg}]
We use \cite[Theorem 4.5]{MR941652}, which gives an explicit formula
for the curvature for each standard vector bundle with its
constant curvature connection. Each such bundle has a height $p\le n$.
When $n=2$, there are three possibilities for $p$. The case $p=0$
just corresponds to a free module with the trivial flat connection.
So we only need to consider the cases $p=1$ and $p=2$. Let
$\bar\partial_1,\,\bar\partial_2$ be a basis for $\fg^{\text{antihol}}$.
The $(0,2)$ component of the curvature, evaluated on
$\bar\partial_1\wedge \bar\partial_2$, is a non-zero constant
multiplied by $\mu(\bar\partial_1\wedge \bar\partial_2 \wedge \Theta)$
when $p=2$, where $\mu$ is a generator of
$\bigwedge^4(\bZ^4)^*$, or is
$\mu(\bar\partial_1\wedge \bar\partial_2)$ when $p=1$, where in this
case $\mu$ is a decomposable element of $\bigwedge^2(\bZ^4)^*$.
Hence to get the desired conclusion we just need to choose
$J$ so that $\mu(\bar\partial_1\wedge \bar\partial_2)\ne 0$
for all $\mu = \alpha\wedge\beta$,
for any linearly independent $\alpha$ and $\beta$ in the lattice
$(\bZ^4)^*$, and then choose $\Th$ so that
$\bar\partial_1\wedge \bar\partial_2 \wedge \Th\ne 0$.
These conditions will be satisfied for ``generic'' $J$ and $\Th$
(not lying in a countable union of proper subvarieties in the
spaces of real $4\times 4$ matrices with $J^2=-1$ and $\Th = - \Th^t$).
\end{proof}
\begin{remark}
It should be obvious that the same method of proof will work for any
complex dimension $n>2$ as well.
\end{remark}  

\subsection{The Case of the Riemann Relations}
\label{sec:Riemann}

The last case we consider will be the one where $J$ satisfies the Riemann
conditions of Theorem \ref{thm:Riemann}.  We will try to give a
weak substitute for Theorem \ref{thm:Riemann1}, showing that
our {\nc} complex torus in this case admits a holomorphic bundle
with ``lots of holomorphic  sections.''
\begin{theorem}
\label{thm:NCRiemann}
Let $(A_\Th, J)$ be a {\nc} complex torus of complex dimension
$n$, and assume that the underlying $(\bZ^{2n}, J)$ admits
a Riemann form $S$.  Then there is a holomorphic bundle $E$
over  $(\cA_\Th, J)$ admitting a flat holomorphic connection,
with $\dim H^0(E)>1$.
\end{theorem}
\begin{proof}
We use some of the same argument that went into the
proof of Theorem \ref{thm:nonalg}.  Namely, we use
\cite[Theorem 4.5]{MR941652} to construct a standard module $E$ of
height $p=1$, but this time we arrange to have
$\mu(\bar\partial_j\wedge \bar\partial_k)= 0$ 
for some $\mu = \alpha\wedge\beta$, $\alpha$ and $\beta$ in the lattice
$\bZ^{2n}$, and for all $j,\,k=1,\cdots,n$.
This corresponds to the construction of an ample line
bundle with curvature of Hodge type $(1,1)$, which is the essence of
the proof in the classical case of Theorem \ref{thm:Riemann1}.

A technical problem here is that we can't just take
pairing with $\mu$ to be given by the Riemann form $S$, because since
$S$ is non-degenerate, it can't be decomposable.  (Any decomposable
$2$-form $\mu$ satisfies $\mu\wedge \mu = 0$, whereas
$\overbrace{S\wedge S\wedge \cdots \wedge S\,}^n\ne 0$.) However,
we can write $S=S_1+\cdots+S_n$ with $S_j$ decomposable for each $j$
and satisfying the same conditions as $S$ except that $S_j$ is the
imaginary part of a positive \emph{semidefinite} hermitian form $H_j$
with kernel of complex dimension $n-1$. (To get the decomposition
of $S$, merely choose a canonical basis of the lattice
$\Lambda\cong\bZ^{2n}$ as in \cite[Definition 3.2.2]{MR3525914},
and let $S_j$ be equal to $S$ on the span of $\nu_j$ and $\nu_{j+n}$
and $0$ on the span of the other basis vectors.)

Now take pairing with $\mu$ to be given by the alternating form $S_1$.
In the Riemann form way of looking at things, we have embedded
our lattice $\Lambda$ in $\fg^{\text{hol}}$, so that the complex
structure is given by multiplication by $i$. The alternating form $S_1$
and hermitian form $H_1$ factor through $\bC\nu_1$ with lattice
spanned by $\nu_1$ and $\nu_{n+1}$, and we are are reduced to the
``splitting case'' of Section \ref{sec:product2}. The conclusion
now follows from the case of complex dimension $1$ and Theorem
\ref{thm:splittingcase}.
\end{proof}
\begin{remark}
Unfortunately, to get the analogue of an ample line bundle
in the situation of Theorem \ref{thm:NCRiemann}, we really should
take the tensor product of the holomorphic bundles corresponding
to each of the $S_j$ (in the notation of the proof), but
we don't have a notion of tensor
product of bundles in the general {\nc} case.  The bundle constructed
in the proof of Theorem \ref{thm:NCRiemann} still has index $0$.

But there is one case where we can get an analogue of an ample line
bundle. If $\Theta$ splits as a block diagonal matrix with $2\times 2$
blocks, in the sense that (in the notation of the proof of Theorem
\ref{thm:NCRiemann}) $\nu_j$ and $\nu_{j+n}$ commute with
$\nu_k$ and $\nu_{k+n}$ for $j\ne k$, then we get a holomorphic map
from a tensor product of {\nc} elliptic curves into $(A_\Th, J)$
as in the product case of Section \ref{sec:product}, and the tensor
product of the holomorphic bundles corresponding to the various
$S_j$'s makes sense.  So in this special case we get something close
to the original theorem of Riemann.
\end{remark}

\bibliographystyle{amsplain}
\bibliography{holo}
\end{document}